\documentclass[12pt]{article} % フォントの大きさとクラスファイルを指定
\usepackage{amscd,amsmath,amssymb,amsfonts,braket,mathrsfs,latexsym,enumerate,amsthm}
\usepackage[all]{xy}
\makeatletter
\@namedef{subjclassname@2010}{%
  \textup{2010} Mathematics Subject Classification}
\makeatother

\newtheorem{thm}{Theorem} % 定理は \begin{thm} \end{thm}
\newtheorem{lem}[thm]{Lemma} % 補題 \begin{lem} \end{lem}

\newtheorem{prop}[thm]{Proposition}

\theoremstyle{definition}

\newtheorem{defn}[thm]{Definition}

\numberwithin{thm}{section}
\numberwithin{equation}{section}

\newcommand{\C}{\mathbb{C}}
\newcommand{\F}{\mathbb{F}}

\newcommand{\Fbar}{\overline{F}}
\newcommand{\Q}{\mathbb{Q}}

\newcommand{\Kb}{\overline{K}}
\newcommand{\kb}{\overline{k}}

\newcommand{\Fsep}{F^{\text{sep}}}
\newcommand{\rhob}{\overline{\rho}}

\newcommand{\Fab}{F^{\text{ab}}}
\newcommand{\Gab}{\mathrm{G}^{\text{ab}}}
\newcommand{\lambdaab}{\lambda^{\text{ab}}}
\newcommand{\trab}{\mathrm{tr}^{\text{ab}}}

\newcommand{\A}{\mathbb{A}}

\newcommand{\kA}{k_{\A}}
\newcommand{\KA}{K_{\A}}

\newcommand{\R}{\mathbb{R}}
\newcommand{\Z}{\mathbb{Z}}

\newcommand{\betab}{\overline{\beta}}

\newcommand{\disc}{\mathrm{disc}\, }

\newcommand{\mfa}{\mathfrak{a}}
\newcommand{\mfp}{\mathfrak{p}}
\newcommand{\mfq}{\mathfrak{q}}
\newcommand{\mfI}{\mathfrak{I}}

\newcommand{\mfl}{\mathfrak{l}}

\newcommand{\cE}{\mathcal{E}}

\newcommand{\cO}{\mathcal{O}}

\newcommand{\cL}{\mathcal{L}}

\newcommand{\cN}{\mathcal{N}}
\newcommand{\cFR}{\mathcal{FR}}

\newcommand{\Ram}{\mathbf{Ram}}

\newcommand{\cMn}{\mathcal{M}^{\textit{new}}}
\newcommand{\cNn}{\mathcal{N}^{\textit{new}}}
\newcommand{\cSn}{\mathcal{S}^{\textit{new}}}
\newcommand{\cTn}{\mathcal{T}^{\textit{new}}}

\newcommand{\id}{\mathrm{id}}

\newcommand{\Frob}{\mathrm{Frob}}

\newcommand{\Gal}{\mathrm{Gal}}
\newcommand{\G}{\mathrm{G}}
\newcommand{\rmH}{\mathrm{H}}
\newcommand{\M}{\mathrm{M}}
\newcommand{\N}{\mathrm{N}}
\newcommand{\GL}{\mathrm{GL}}

\newcommand{\End}{\mathrm{End}}
\newcommand{\Aut}{\mathrm{Aut}}

\newcommand{\im}{\mathrm{Im}\, }
\newcommand{\Norm}{\mathrm{Norm}}

\newcommand{\ord}{\mathrm{ord}\, }

\newcommand{\tr}{\mathrm{tr}}

\newcommand{\Spec}{\mathrm{Spec}}

\newcommand{\cf}{cf.\ }
\newcommand{\inj}{\hookrightarrow}

\newcommand{\resp}{resp.\ }

\newcommand{\ch}{\mathrm{char}\,}

\begin{document}

\title{An effective bound of $p$ for algebraic points on Shimura curves of $\Gamma_0(p)$-type}
\author{Keisuke Arai}
\date{}
%\address[Keisuke Arai]
%{Department of Mathematics, School of Engineering,
%Tokyo Denki University,
%2-2 Kanda-Nishiki-cho, Chiyoda-ku, Tokyo, Japan 
%101-8457}
%\email{araik@mail.dendai.ac.jp}

%\address[Fumiyuki Momose]
%{Department of Mathematics, Faculty of Science and Engineering,
%Chuo University,
%1-13-27 Kasuga, Bunkyo-ku, Tokyo 112-8551, Japan}
%\email{momose@math.chuo-u.ac.jp}

%\date{}

%\pagestyle{plain}
%
%

%
%

%\subjclass[2010]{Primary 11G18, 14G05; Secondary 11G10, 11G15}

%\keywords{rational points, Shimura curves, QM-abelian surfaces}

\maketitle

%\begin{center}
%\textit{To the memory of Fumiyuki Momose}
%\end{center}

\begin{abstract}

%In this article we study algebraic points on Shimura curves
%of $\Gamma_0(p)$-type.
In previous articles,
we classified the characters associated to algebraic points
on Shimura curves of $\Gamma_0(p)$-type,
and over number fields
in a certain large class
we showed that
there are at most elliptic points on such a Shimura curve
for every sufficiently large prime number $p$.
In this article, we obtain
an effective bound of $p$ concerning % the classification of the characters and 
algebraic points on Shimura curves of $\Gamma_0(p)$-type.

%As an application we show the irreducibility of the mod-$p$ Galois representation
%associated to a QM-abelian surface, together with some finiteness on abelian varieties.

\end{abstract}

%\noindent
%2010 \textit{Mathematics Subject Classification.}
%Primary 11G18, 14G05; Secondary 11G10, 11G15.

%\tableofcontents

\vspace{5mm}
\noindent
{\bf Notation}

\vspace{1mm}

For an integer $n\geq 1$ and a commutative group (or a commutative group scheme) $G$,
let $G[n]$ denote the kernel of multiplication by $n$ in $G$.
For a field $F$,
let $\ch F$ denote the characteristic of $F$,
let $\Fbar$ denote an algebraic closure of $F$,
let $\Fsep$ (\resp $\Fab$) denote the separable closure
(\resp the maximal abelian extension) of $F$ inside $\Fbar$,
and let $\G_F=\Gal(\Fsep/F)$, $\Gab_F=\Gal(\Fab/F)$.
For a prime number $p$ and a field $F$ of characteristic $0$, let
$\theta_p:\G_F\longrightarrow\F_p^{\times}$ denote the mod $p$ cyclotomic character.
Let $|\cdot|$ denote the usual complex absolute value on $\C$.
For a number field or a local field $k$, let $\cO_k$ denote
the ring of integers of $k$.
For a number field $k$,
put $n_k:=[k:\Q]$;
let $d_k$ denote the discriminant of $k$;
let $h_k$ denote the class number of $k$;
let $r_k$ denote the rank of the unit group $\cO_k^{\times}$;
let $R_k$ denote the regulator of $k$;
fix an inclusion $k\hookrightarrow\C$
and take the algebraic closure $\kb$ inside $\C$;
let $k_v$ denote the completion of $k$ at $v$
where $v$ is a place (or a prime) of $k$;
let $k_{\A}$ denote the ad\`{e}le ring of $k$;
and let $\Ram (k)$ denote the set of prime numbers which are ramified in $k$.
For a scheme $S$ and an abelian scheme $A$ over $S$,
let $\End_S(A)$ denote the ring of endomorphisms of $A$ defined over $S$.
If $S=\Spec(F)$ for a field $F$ and if $F'/F$ is a field extension,
simply put
$\End_{F'}(A):=\End_{\Spec(F')}(A\times_{\Spec(F)}\Spec(F'))$
and
$\End(A):=\End_{\Fbar}(A)$.
Let $\Aut(A):=\Aut_{\Fbar}(A)$ be the group
of automorphisms of $A$ defined over $\Fbar$.
For a prime number $p$ and an abelian variety $A$ over a field $F$,
let
$\displaystyle T_pA:=\lim_{\longleftarrow}A[p^n](\Fbar)$
be the $p$-adic Tate module of $A$,
where the inverse limit is taken with respect to
multiplication by $p$ : $A[p^{n+1}](\Fbar)\longrightarrow A[p^n](\Fbar)$.

\section{Introduction}
\label{intro}

Let $B$ be an indefinite quaternion division algebra over $\Q$.
Let
$$d:=\disc B$$
be the discriminant of $B$.
Then $d$ is the product of an even number of distinct prime numbers, and $d>1$.
Fix a maximal order $\cO$ of $B$.
For each prime number $p$ not dividing $d$, fix an isomorphism
\begin{equation}
\label{OM2}
\cO\otimes_{\Z}\Z_p\cong\M_2(\Z_p) %\nonumber
%\leqno(1.1)
\end{equation}
of $\Z_p$-algebras.

\begin{defn}
\label{defqm}
%\rm

(\cf \cite[p.591]{Bu})
Let $S$ be a scheme.
%where $d$ is invertible.
A QM-abelian surface by $\cO$
over $S$ is a pair $(A,i)$ where $A$ is an 
abelian surface over $S$ (i.e. $A$ is an abelian scheme over $S$ 
of relative dimension $2$), and 
$i:\cO\inj\End_S(A)$ 
is an injective ring homomorphism (sending $1$ to $\id$). 
We assume that $A$ has a left $\cO$-action.
We will sometimes omit ``by $\cO$" and simply write ``a QM-abelian surface"
if there is no fear of confusion.

\end{defn}

%\noindent
Let $M^B$ be the coarse moduli scheme over $\Q$ parameterizing isomorphism classes
of QM-abelian surfaces by $\cO$.
The notation $M^B$ is permissible
although we should write $M^{\cO}$ instead of $M^B$;
for even if we replace $\cO$ by another maximal order $\cO'$,
we have a natural isomorphism
$M^{\cO}\cong M^{\cO'}$
since $\cO$ and $\cO'$ are conjugate in $B$.
Then $M^B$ is a proper smooth curve over $\Q$, called a Shimura curve.
For a prime number $p$ not dividing $d$,
let $M_0^B(p)$
% (or simply $M_0(p)$) 
be the coarse moduli scheme over $\Q$ parameterizing isomorphism classes
of triples $(A,i,V)$ where $(A,i)$ is a QM-abelian surface by $\cO$
and $V$ is a left $\cO$-submodule of $A[p]$ with $\F_p$-dimension $2$.
%Here $A[p]$ is the kernel of multiplication by $p$ in $A$.
Then $M_0^B(p)$ is a proper smooth curve over $\Q$, which we call a
Shimura curve of $\Gamma_0(p)$-type.
We have a natural map $$\pi^B(p):M_0^B(p)\longrightarrow M^B$$
over $\Q$ defined by $(A,i,V)\longmapsto (A,i)$.

For real points on $M^B$, we know the following.

\begin{thm}[{\cite[Theorem 0, p.136]{Sh}}]
\label{M^B(R)}

We have $M^B(\R)=\emptyset$.

\end{thm}

In the previous article, we showed that there are few points on $M_0^B(p)$
for every sufficiently large prime number $p$,
which is an analogue of the study of points on the modular curve $X_0(p)$
(\cite{Ma}, \cite{Mo};
for related topics, see \cite{A2}).

\begin{thm}[{\cite[Theorem 1.4]{A3}}]
\label{prevthm}

Let $k$ be a finite Galois extension of $\Q$ which does not contain
the Hilbert class field of any imaginary quadratic field.
Then there is a finite set $\cL(k)$
of prime numbers
depending on $k$
which satisfies the following.
Assume that there is a prime number $q$ which splits completely in $k$ and satisfies
$B\otimes_{\Q}\Q(\sqrt{-q})\not\cong\M_2(\Q(\sqrt{-q}))$, and
let $p>4q$ be a prime number which also satisfies
$p\geq 11$, $p\ne 13$, $p\nmid d$ and $p\not\in\cL(k)$.

(1)
If $B\otimes_{\Q}k\cong\M_2(k)$, then $M_0^B(p)(k)=\emptyset$.

(2)
If $B\otimes_{\Q}k\not\cong\M_2(k)$, then
$M_0^B(p)(k)\subseteq\{\text{elliptic points of order $2$ or $3$}\}$.

\end{thm}

We can identify $M_0^B(p)(\C)$ with a quotient of the upper half-plane,
and we use the notion of "elliptic points" in this context.
In Theorem \ref{prevthm}, the set $\cL(k)$ can be described explicitly.
In this article, we modify the set $\cL(k)$ slightly, and obtain
its effective bound in Section \ref{est}.

\section{Galois representations associated to QM-abelian surfaces (generalities)}
\label{QM}

We consider the Galois representation associated to a QM-abelian surface.
Take a prime number $p$ not dividing $d$.
Let $F$ be a field with $\ch F\ne p$.
Let $(A,i)$ be a QM-abelian surface by $\cO$ over $F$.
We have isomorphisms of $\Z_p$-modules:
$$\Z_p^4\cong T_pA\cong\cO\otimes_{\Z}\Z_p\cong\M_2(\Z_p).$$
The middle is also an isomorphism of left $\cO$-modules;
the last is also an isomorphism of $\Z_p$-algebras (which is fixed in (\ref{OM2})).
We sometimes identify these $\Z_p$-modules.
Take a $\Z_p$-basis
$$e_1=\left(\begin{matrix} 1&0 \\ 0&0\end{matrix}\right),\ 
e_2=\left(\begin{matrix} 0&0 \\ 1&0\end{matrix}\right),\ 
e_3=\left(\begin{matrix} 0&1 \\ 0&0\end{matrix}\right),\ 
e_4=\left(\begin{matrix} 0&0 \\ 0&1\end{matrix}\right)$$
of $\M_2(\Z_p)$.
Then the image of the natural map
$$\M_2(\Z_p)\cong\cO\otimes_{\Z}\Z_p\hookrightarrow\End(T_pA)\cong\M_4(\Z_p)$$
lies in
$\left\{\left(\begin{matrix} 
X\ &0 \\ 0\ &X
\end{matrix}\right)
\Bigg|\, X\in\M_2(\Z_p)\right\}$.
%
%Let $\Fsep$ be a separable closure of $F$, and
%let $\G_F=\Gal(\Fsep/F)$ be the absolute Galois group of $F$.
The action of the Galois group $\G_F$ on $T_pA$ induces a representation
$$\rho:\G_F\longrightarrow\Aut_{\cO\otimes_{\Z}\Z_p}(T_pA)\subseteq\Aut(T_pA)
\cong\GL_4(\Z_p),$$
where
$\Aut_{\cO\otimes_{\Z}\Z_p}(T_pA)$
is the group of automorphisms of $T_pA$
commuting with the action of $\cO\otimes_{\Z}\Z_p$.
We often identify
$\Aut(T_pA)=\GL_4(\Z_p)$.
The above observation implies
$$\Aut_{\cO\otimes_{\Z}\Z_p}(T_pA)=
\left\{\left(\begin{matrix} sI_2&tI_2 \\ uI_2&vI_2\end{matrix}\right) \Bigg|\,
\left(\begin{matrix} s&t \\ u&v\end{matrix}\right)\in\GL_2(\Z_p)\right\},$$
where 
$I_2=\left(\begin{matrix} 1\ &0 \\ 0\ &1\end{matrix}\right)$.
Then the Galois representation $\rho$ factors as
$$\rho:\G_F\longrightarrow
\left\{\left(\begin{matrix} sI_2&tI_2 \\ uI_2&vI_2\end{matrix}\right) \Bigg|\,
\left(\begin{matrix} s&t \\ u&v\end{matrix}\right)\in\GL_2(\Z_p)\right\}
\subseteq\GL_4(\Z_p).$$
Let 
$$\rhob:\G_F\longrightarrow
\left\{\left(\begin{matrix} sI_2&tI_2 \\ uI_2&vI_2\end{matrix}\right) \Bigg|\,
\left(\begin{matrix} s&t \\ u&v\end{matrix}\right)\in\GL_2(\F_p)\right\}
\subseteq\GL_4(\F_p)$$
be the reduction of $\rho$ modulo $p$.
Let
\begin{equation}
\label{rhobar}
\rhob_{A,p}:\G_F\longrightarrow\GL_2(\F_p) %\nonumber
%\leqno(2.1)
\end{equation}
denote the Galois representation induced from $\rhob$ by
$``\left(\begin{matrix} s&t \\ u&v\end{matrix}\right)"$,
so that we have
$\rhob_{A,p}(\sigma)=\left(\begin{matrix} s&t \\ u&v\end{matrix}\right)$
if $\rhob(\sigma)=\left(\begin{matrix} sI_2&tI_2 \\ uI_2&vI_2\end{matrix}\right)$
for $\sigma\in\G_F$.

%

%For a separable closure $\Fsep$ of $F$,
Suppose that $A[p](\Fsep)$ has a left $\cO$-submodule $V$ with $\F_p$-dimension $2$
which is stable under the action of $\G_F$.
We may assume
$V=\F_p e_1\oplus\F_p e_2
=\left\{\left(\begin{matrix} *&0 \\ *&0 \end{matrix}\right)\right\}$.
Since $V$ is stable under the action of $\G_F$, we find
$\rhob_{A,p}(\G_F)\subseteq
\left\{\left(\begin{matrix} 
s&t \\ 0&v
\end{matrix}\right)\right\}\subseteq\GL_2(\F_p)$.
%
%We also find that $\G_F$ acts on $V$ as
%``multiplication by $x$";
Let
\begin{equation}
\label{lambda}
\lambda:\G_F\longrightarrow\F_p^{\times} %\nonumber
%\leqno(2.2)
\end{equation}
denote the character induced from $\rhob_{A,p}$ by ``$s$", so that
$\rhob_{A,p}(\sigma)=
\left(\begin{matrix} 
\lambda(\sigma)&* \\ 0&*
\end{matrix}\right)$
for $\sigma\in\G_F$.
Note that
$\G_F$ acts on $V$ by $\lambda$
(i.e. $\rhob(\sigma)(v)=\lambda(\sigma)v$
for $\sigma\in\G_F$, $v\in V$).

\section{Automorphism groups}
\label{Aut}

We consider the automorphism group of a QM-abelian surface.
Let $(A,i)$ be a QM-abelian surface by $\cO$ over a field $F$.
Put
$$\End_{\cO}(A):=\{f\in\End(A)\mid fi(g)=i(g)f \text{\ \ for any $g\in\cO$}\}$$
and
$$\Aut_{\cO}(A):=\Aut(A)\cap\End_{\cO}(A).$$
If $\ch F=0$,
then $\Aut_{\cO}(A)\cong\Z/2\Z$, $\Z/4\Z$ or $\Z/6\Z$.
%The isomorphism $\Aut_{\cO}(A)\cong\Z/4\Z$ (\resp $\Aut_{\cO}(A)\cong\Z/6\Z$)
%is possible only when all the prime divisors of $d$ are congruent to $2$ or $-1$ modulo $4$
%(\resp $0$ or $-1$ modulo $3$).

Let $p$ be a prime number not dividing $d$.
Let $(A,i,V)$ be a triple where $(A,i)$ is a QM-abelian surface by $\cO$ over
a field $F$
and $V$ is a left $\cO$-submodule of $A[p](\Fbar)$ with $\F_p$-dimension $2$.
Define a subgroup $\Aut_{\cO}(A,V)$ of $\Aut_{\cO}(A)$ by
$$\Aut_{\cO}(A,V):=\{f\in\Aut_{\cO}(A)\mid f(V)=V\}.$$
Assume $\ch F=0$.
Then $\Aut_{\cO}(A,V)\cong\Z/2\Z$, $\Z/4\Z$ or $\Z/6\Z$.
Notice that we have
$\Aut_{\cO}(A)\cong\Z/2\Z$
(\resp $\Aut_{\cO}(A,V)\cong\Z/2\Z$)
if and only if
$\Aut_{\cO}(A)=\{\pm 1\}$
(\resp $\Aut_{\cO}(A,V)=\{\pm 1\}$).

\section{Fields of definition}
\label{fieldofdefinition}

Let $k$ be a number field. %(considered as a subfield of $\C$).
Let $p$ be a prime number not dividing $d$.
Take a point
$$x\in M_0^B(p)(k).$$
Let $x'\in M^B(k)$ be the image
of $x$ by the map $\pi^B(p):M_0^B(p)\longrightarrow M^B$.
Then $x'$ is represented by a QM-abelian surface (say $(A_x,i_x)$) over $\kb$,
and $x$ is represented by a triple $(A_x,i_x,V_x)$
where $V_x$ is a left $\cO$-submodule of $A[p](\kb)$
with $\F_p$-dimension $2$.
For a finite extension $M$ of $k$ (in $\kb$),
we say that we can take $(A_x,i_x)$ (\resp $(A_x,i_x,V_x)$) to be defined
over $M$
if there is a QM-abelian surface $(A,i)$ over $M$ such that
$(A,i)\otimes_M\kb$ is isomorphic to $(A_x,i_x)$
(\resp if there is a QM-abelian surface $(A,i)$ over $M$
and a left $\cO$-submodule $V$ of $A[p](\kb)$ with $\F_p$-dimension $2$
stable under the action of $\G_M$ such that there is an isomorphism
between $(A,i)\otimes_M\kb$ and $(A_x,i_x)$ under which $V$ corresponds
to $V_x$).
Put
$$\Aut(x):=\Aut_{\cO}(A_x,V_x),\ \ \ \ \Aut(x'):=\Aut_{\cO}(A_x).$$
Then $\Aut(x)$ is a subgroup of $\Aut(x')$.
%
%The point $x$ is called a CM point if $A_x$ has complex multiplication.
%Note that if $\Aut(x')\cong\Z/4\Z$ or $\Z/6\Z$,
%then $x$ is a CM point.
%
%In particular, if $x$ is an elliptic point of order $2$ or $3$,
%then $x$ is a CM point.
%
Note that $x$ is an elliptic point of order $2$ (\resp $3$)
if and only if
$\Aut(x)\cong\Z/4\Z$
(\resp $\Aut(x)\cong\Z/6\Z$).

Since $x$ is a $k$-rational point, we have 
$^{\sigma}x=x$ for any $\sigma\in\G_k$.
Then, for any $\sigma\in\G_k$, there is an isomorphism
$$\phi_{\sigma}:{}^{\sigma}(A_x,i_x,V_x)\longrightarrow (A_x,i_x,V_x),$$
which we fix once for all.
Let
$$\phi'_{\sigma}:{}^{\sigma}(A_x,i_x)\longrightarrow (A_x,i_x)$$
be the isomorphism induced from $\phi_{\sigma}$ by forgetting $V_x$.
For $\sigma,\tau\in\G_k$, put
$$c_x(\sigma,\tau)
:=\phi_{\sigma}\circ {}^{\sigma}\phi_{\tau}
\circ\phi_{\sigma\tau}^{-1}\in\Aut(x)$$
and
$$c'_x(\sigma,\tau)
:=\phi'_{\sigma}\circ {}^{\sigma}\phi'_{\tau}
\circ(\phi'_{\sigma\tau})^{-1}\in\Aut(x').$$
Then $c_x$ (\resp $c'_x$) is a $2$-cocycle
and defines a cohomology class
$[c_x]\in H^2(\G_k,\Aut(x))$
(\resp $[c'_x]\in H^2(\G_k,\Aut(x'))$.
Here the action of $\G_k$ on $\Aut(x)$ (\resp $\Aut(x')$)
is defined in a natural manner (\cf \cite[Section 4]{AM}).

\begin{prop}[{\cite[Theorem (1.1), p.93]{J}}]
\label{fieldMB}

We can take $(A_x,i_x)$ to be defined over $k$
if and only if
$B\otimes_{\Q}k\cong\M_2(k)$.

\end{prop}

\begin{prop}[{\cite[Proposition 4.2]{AM}}]
\label{fieldM0Bp}

(1)
Suppose $B\otimes_{\Q}k\cong\M_2(k)$.
Further assume $\Aut(x)\ne\{\pm 1\}$ or
$\Aut(x')\not\cong\Z/4\Z$.
Then we can take $(A_x,i_x,V_x)$ to be defined over $k$.

(2)
Assume $\Aut(x)=\{\pm 1\}$.
Then there is a quadratic extension $K$ of $k$
such that we can take $(A_x,i_x,V_x)$ to be defined over $K$.

\end{prop}

\begin{lem}[{\cite[Lemma 4.3]{AM}}]
\label{fieldofdef}

Let $K$ be a quadratic extension of $k$.
Assume $\Aut(x)=\{\pm 1\}$.
Then the following two conditions are equivalent.

(1)
We can take $(A_x,i_x,V_x)$ to be defined over $K$.

(2)
For any place $v$ of $k$ satisfying $[c_x]_v\ne 0$,
the tensor product $K\otimes_k k_v$ is a field.
%(i.e. $v$ does not split in $K/k$).

\end{lem}

\section{Classification of characters}
\label{charI}

We keep the notation in Section \ref{fieldofdefinition}.
Throughout this section,
assume $\Aut(x)=\{\pm 1\}$.
Let $K$ be a quadratic extension of $k$
which satisfies the equivalent conditions in
Lemma \ref{fieldofdef}.
%If $B\otimes_{\Q}k\cong\M_2(k)$, we impose no restrictions
%on the quadratic extension $K$.
%Note that the quadratic extension $K$ has no restriction
%if $B\otimes_{\Q}k\cong\M_2(k)$.
Then $x$ is represented by a triple $(A,i,V)$,
where $(A,i)$ is a QM-abelian surface over $K$ and $V$ is a left
$\cO$-submodule of $A[p](\Kb)$ with $\F_p$-dimension $2$
stable under the action of $\G_K$.
Let
$$\lambda:\G_K\longrightarrow\F_p^{\times}$$
be the character associated to $V$ in (\ref{lambda}).
For a prime $\mfl$ of $k$ (\resp $K$), let $I_{\mfl}$ denote the inertia
subgroup of $\G_k$ (\resp $\G_K$) at $\mfl$.

Let $\lambdaab:\Gab_K\longrightarrow\F_p^{\times}$
be the natural map induced from $\lambda$.
Put
\begin{equation}
\label{phi}
\varphi:=\lambdaab\circ\tr_{K/k}:\G_k\longrightarrow\Gab_K
\longrightarrow\F_p^{\times} %\nonumber
%\leqno(5.1)
\end{equation}
where $\tr_{K/k}:\G_k\longrightarrow\Gab_K$ is the transfer map.
Notice that the induced map
$\trab_{K/k}:\Gab_k\longrightarrow\Gab_K$
from $\tr_{K/k}$ corresponds to the natural inclusion
$\kA^{\times}\hookrightarrow\KA^{\times}$
via class field theory
(\cite[Theorem 8 in \S 9 of Chapter XIII, p.276]{We}).
We know that $\varphi^{12}$ is unramified at every prime of $k$
not dividing $p$ (\cite[Corollary 5.2]{AM}), and so
$\varphi^{12}$ corresponds to a character of the
ideal group $\mfI_k(p)$ consisting of
fractional ideals of $k$ prime to $p$.
By abuse of notation, let $\varphi^{12}$ also denote
the corresponding character
of $\mfI_k(p)$.

Let $\cMn$ be the set of prime numbers $q$ such that $q$ splits
completely in $k$.
%and $q$ does not divide $6h_k$.
Let $\cNn$ be the set of primes $\mfq$ of $k$ such that
$\mfq$ divides some prime number $q\in \cMn$.
%Let $\mfI_k$ be the ideal group of $k$, and let $\mfP_k$
%be the subgroup of $\mfI_k$ consisting of principal ideals.
Take a finite subset $\emptyset\ne \cSn\subseteq \cNn$
which generates the ideal class group of $k$. %$Cl_k=\mfI_k/\mfP_k$.
For each prime $\mfq\in \cSn$, fix an element $\alpha_{\mfq}\in\cO_k\setminus\{0\}$
satisfying $\mfq^{h_k}=\alpha_{\mfq}\cO_k$.

For an integer $n\geq 1$, put
$$\cFR(n):=\Set{\beta\in\C|
\beta^2+a\beta+n=0 \text{ for some integer $a\in\Z$ with $|a|\leq 2\sqrt{n}$}}.$$
For any element $\beta\in\cFR(n)$, we have $|\beta|=\sqrt{n}$.
%
%Notice that $|a|\leq 2\sqrt{q}$ implies $|a|<2\sqrt{q}$
%since $2\sqrt{q}$ is not a rational number.
%
For a prime $\mfq$ of $k$,
put $\N(\mfq)=\sharp(\cO_k/\mfq)$.
If $\mfq\in\cSn$,
then $\N(\mfq)$ is a prime number.
Define the sets

\noindent
$\cE(k):=\Set{\varepsilon_0=\displaystyle\sum_{\sigma\in\Gal(k/\Q)}a_{\sigma}\sigma
\in\Z[\Gal(k/\Q)]|
a_{\sigma}\in\{0,8,12,16,24 \}}$,

\noindent
$\cMn_1(k):=
\Set{(\mfq,\varepsilon_0,\beta_{\mfq})|
\mfq\in \cSn,\ \varepsilon_0
\in\cE(k),\ 
\beta_{\mfq}\in\cFR(\N(\mfq))}$,

\noindent
$\cMn_2(k):=\Set{\Norm_{k(\beta_{\mfq})/\Q}(\alpha_{\mfq}^{\varepsilon_0}-\beta_{\mfq}^{24h_k})\in\Z|
(\mfq,\varepsilon_0,\beta_{\mfq})\in\cMn_1(k)}\setminus\{0\}$,

\noindent
$\cNn_0(k):=\Set{\text{$l$ : prime number}|\text{$l$ divides some integer $m\in\cMn_2(k)$}}$,

\noindent
$\cTn(k):=\Set{\text{$l'$ : prime number}|\text{$l'$ is divisible
by some prime $\mfq'\in \cSn$}}
\cup\{2,3\}$,

\noindent
$\cNn_1(k):=\cNn_0(k)\cup\cTn(k)\cup\Ram(k)$.

\noindent
Note that all the sets, $\cFR(n), \cMn_1(k)$, $\cMn_2(k)$, $\cNn_0(k)$, $\cTn(k)$,
and $\cNn_1(k)$, are finite.

\begin{thm}[\cf {\cite[Theorem 5.6]{AM}}]
\label{type23phi}

Assume that $k$ is Galois over $\Q$.
If $p\not\in\cNn_1(k)$
(and if $p$ does not divide $d$),
then the character
$\varphi:\G_k\longrightarrow \F_p^{\times}$
%associated to a non-CM point of $M_0^B(p)(k)$ 
is of one of the following types.

Type 2:
$\varphi^{12}=\theta_p^{12}$ and $p\equiv 3\bmod{4}$.

Type 3:
There is an imaginary quadratic field $L$ satisfying the following
two conditions.

%\noindent
%(a)
%The prime number $p$ splits in $L$.

\noindent
(a)
The Hilbert class field $H_L$ of $L$ is contained in $k$.

\noindent
(b)
There is a prime $\mfp_L$ of $L$ lying over $p$
such that
$\varphi^{12}(\mfa)\equiv\delta^{24}\bmod{\mfp_L}$ holds
for any fractional ideal $\mfa$ of $k$ prime to $p$.
Here $\delta$ is any element of $L$ such that
$\Norm_{k/L}(\mfa)=\delta\cO_L$.

\end{thm}

\begin{proof}

It suffices to modify the proof of Theorem 5.6 in \cite{AM} slightly.
First, notice that the abelian surface
$A\otimes_K K_{\mfq_K}$ over $K_{\mfq_K}$
(which corresponds to $x\otimes K_{\mfq_K}$)
has good reduction after
a totally ramified finite extension $M(\mfq)/K_{\mfq_K}$
without assuming $q\geq 5$
(\cite[Proposition 3.2, p.101]{J}).
Second, in the case where $\varepsilon$ is of type 2,
notice that $\beta_{\mfq}^{24h_k}=q^{12h_k}$
implies $\beta_{\mfq}^{24}=q^{12}$
without assuming $q\nmid 6h_k$.
Now we show this.
Write $\beta=\beta_{\mfq}$ for simplicity.
Let $\betab$ be the complex conjugate of $\beta$.
Since $\beta^{24h_k}=\betab^{24h_k}$, we have
$\betab=\zeta\beta$ for some $\zeta\in\C$ with $\zeta^{24h_k}=1$.
Since $\Q(\beta)=\Q(\betab)=\Q(\zeta\beta)=\Q(\beta, \zeta)\supseteq\Q(\zeta)$
and $[\Q(\beta):\Q]=2$,
we have $\zeta^4=1$ or $\zeta^6=1$.
Then $\zeta^{12}=1$.
This implies $\betab^{12}=\zeta^{12}\beta^{12}=\beta^{12}$,
and so $\beta^{12}\in\Q$.
Since $|\beta|=\sqrt{q}$, we have $\beta^{12}=\pm q^6$,
and hence $\beta^{24}=q^{12}$.
Notice that the case $\beta^{12}=-q^6$ really occurs
(e.g. $q=2$ and $\beta=1+\sqrt{-1}$).

\end{proof}

From now to the end of this section, assume that $k$ is Galois over $\Q$.
The set $\cNn_1(k)$ may differ from $\cN_1(k)$ in \cite{AM},
nevertheless we have the following.

\begin{lem}[\cf {\cite[Lemma 5.11]{AM2}}]
\label{type2lambda}

Suppose $p\geq 11$, $p\ne 13$ and $p\not\in\cNn_1(k)$.
Further assume the following two conditions.

(a)
Every prime $\mfp$ of $k$ above $p$ is inert in $K/k$.

(b)
Every prime $\mfq\in\cSn$ is ramified in $K/k$.

\noindent
If $\varphi$ is of type 2, then we have the following.

(i) The character $\lambda^{12}\theta_p^{-6}:\G_K\longrightarrow\F_p^{\times}$
is unramified everywhere.

(ii) The map $Cl_K\longrightarrow\F_p^{\times}$ induced from $\lambda^{12}\theta_p^{-6}$
is trivial on
$C_{K/k}:=\im(Cl_k\longrightarrow Cl_K)$,
where $Cl_K$ is the ideal class group of $K$
and
$Cl_k\longrightarrow Cl_K$
is the map defined by
$[\mfa]\longmapsto[\mfa\cO_K]$.

\end{lem}

\begin{proof}

(ii)
The argument in the proof of Lemma 5.11 (ii) in \cite{AM2} works here, but we may not have
$\beta=\pm\sqrt{-q}$.
Nevertheless we have $\beta^{24}=q^{12}$ as seen in
the proof of Theorem \ref{type23phi}, and so we are done.

\end{proof}

We also have the following with the same proof as before.

\begin{lem}[\cf {\cite[Lemma 5.6]{A3}, \cite[Lemma 5.12]{AM}, \cite[3]{AM2}}]
\label{q/p-1}

Suppose $p\geq 11$, $p\ne 13$ and $p\not\in\cNn_1(k)$.
Assume that $\varphi$ is of type 2.
Let $q<\frac{p}{4}$ be a prime number
which splits completely in $k$.
Then we have
$\left(\frac{q}{p}\right)=-1$
and
$q^{\frac{p-1}{2}}\equiv -1\bmod{p}$.
Furthermore, we have
$B\otimes_{\Q}\Q(\sqrt{-q})\cong\M_2(\Q(\sqrt{-q}))$.

\end{lem}

%\begin{lem}
%\label{q^p-1/2}

%In the situation of Lemma \ref{q/p-1},
%we have
%$q^{\frac{p-1}{2}}\equiv -1\bmod{p}$.

%\end{lem}

%\begin{proof}

%Since $(q^{\frac{p-1}{2}})^2=q^{p-1}\equiv 1\bmod{p}$,
%we have
%$q^{\frac{p-1}{2}}\equiv \pm 1\bmod{p}$.
%Assume $q^{\frac{p-1}{2}}\equiv 1\bmod{p}$.
%Then
%$(q^{\frac{p+1}{4}})^2=q^{\frac{p+1}{2}}\equiv q\bmod{p}$.
%Then $\left(\frac{q}{p}\right)=1$,
%which contradicts $\left(\frac{q}{p}\right)=-1$.
%Therefore $q^{\frac{p-1}{2}}\equiv -1\bmod{p}$.

%\end{proof}

\section{Irreducibility result}
\label{charII}

Let $k$ be a number field, and
let $(A,i)$ be a QM-abelian surface by $\cO$ over $k$.
For a prime number $p$ not dividing $d$,
assume that the representation $\rhob_{A,p}$ in (\ref{rhobar}) is reducible.
Then there is a 1-dimensional sub-representation of $\rhob_{A,p}$;
let $\nu$ be its associated character.
In this case notice that there is a left $\cO$-submodule $V$
of $A[p](\kb)$ with $\F_p$-dimension $2$ on which $\G_k$ acts by $\nu$,
and so the triple $(A,i,V)$ determines a point $x\in M_0^B(p)(k)$.
Take any quadratic extension $K$ of $k$.
Then we have the characters
$\lambda:\G_K\longrightarrow\F_p^{\times}$
and
$\varphi:\G_k\longrightarrow\F_p^{\times}$
associated to the triple $(A\otimes_k K,i,V)$.
We know $\varphi=\nu^2$ by construction of $\varphi$.

\begin{thm}
\label{irred}

Let $k$ be a finite Galois extension of $\Q$ which does not contain
the Hilbert class field of any imaginary quadratic field.
Assume that there is a prime number $q$ which splits completely in $k$
and satisfies $B\otimes_{\Q}\Q(\sqrt{-q})\not\cong\M_2(\Q(\sqrt{-q}))$.
Let $p>4q$ be a prime number which also satisfies $p\geq 11$, $p\ne 13$,
$p\nmid d$ and $p\not\in\cNn_1(k)$.
Then the representation
$$\rhob_{A,p}:\G_k\longrightarrow\GL_2(\F_p)$$
is irreducible.

\end{thm}

\begin{proof}

Assume that $\rhob_{A,p}$ is reducible.
Then the associated character $\varphi$ is of type 2 in Theorem \ref{type23phi},
because $k$ does not contain
the Hilbert class field of any imaginary quadratic field.
By Lemma \ref{q/p-1}, we have
$B\otimes_{\Q}\Q(\sqrt{-q})\cong\M_2(\Q(\sqrt{-q}))$,
which is a contradiction.

\end{proof}

\begin{thm}
\label{mainthm}

Let $k$ be a finite Galois extension of $\Q$ which does not contain
the Hilbert class field of any imaginary quadratic field.
%Then there is an effectively computable finite set $\cL(k)$
%of prime numbers satisfying the following.
Assume that there is a prime number $q$ which splits completely in $k$ and satisfies
$B\otimes_{\Q}\Q(\sqrt{-q})\not\cong\M_2(\Q(\sqrt{-q}))$, and
let $p>4q$ be a prime number which also satisfies
$p\geq 11$, $p\ne 13$, $p\nmid d$ and $p\not\in\cNn_1(k)$.

(1)
If $B\otimes_{\Q}k\cong\M_2(k)$, then $M_0^B(p)(k)=\emptyset$.

(2)
If $B\otimes_{\Q}k\not\cong\M_2(k)$, then
$M_0^B(p)(k)\subseteq\{\text{elliptic points of order $2$ or $3$}\}$.

\end{thm}

\noindent
(Proof of Theorem \ref{mainthm})

Let $k$ be a finite Galois extension of $\Q$ which does not contain
the Hilbert class field of any imaginary quadratic field,
and let $q$ be a prime number which splits completely in $k$
and satisfies $B\otimes_{\Q}\Q(\sqrt{-q})\not\cong\M_2(\Q(\sqrt{-q}))$.
Let $p>4q$ be a prime number which also satisfies
$p\geq 11$, $p\ne 13$ and $p\nmid d$.
Take a point $x\in M_0^B(p)(k)$.

(1)
Suppose $B\otimes_{\Q}k\cong\M_2(k)$.

(1-i)
Assume $\Aut(x)\ne\{\pm 1\}$
or $\Aut(x')\not\cong\Z/4\Z$.
Then $x$ is represented by a triple $(A,i,V)$ defined over $k$
by Proposition \ref{fieldM0Bp} (1),
and the representation $\rhob_{A,p}$ is reducible.
By Theorem \ref{irred}, we have $p\in\cNn_1(k)$.

(1-ii)
Assume otherwise (i.e. $\Aut(x)=\{\pm 1\}$
and $\Aut(x')\cong\Z/4\Z$).
Then $x$ is represented by a triple $(A,i,V)$ defined over a quadratic extension of $k$
by Proposition \ref{fieldM0Bp} (2),
and we have a character
$\varphi:\G_k\longrightarrow\F_p^{\times}$ as in (\ref{phi}).
By Theorem \ref{type23phi} and Lemma \ref{q/p-1},
we have $p\in\cNn_1(k)$.

(2)
Suppose $B\otimes_{\Q}k\not\cong\M_2(k)$.
Further assume that $x$ is not an elliptic point of order $2$ or $3$;
this implies $\Aut(x)=\{\pm 1\}$.
By the same argument as in (1-ii),
we conclude $p\in\cNn_1(k)$.

\qed

\section{Estimate of $\cNn_1(k)$}
\label{est}

We estimate the set $\cNn_1(k)$ in the method of \cite{Da}.

\begin{thm}[{\cite[Theorem 1.1, p.272]{LMO}}]
\label{effectivechebotarev}

There is an absolute, effectively computable constant $A_1>1$ such that for every
finite extension $k_1$ of $\Q$, every finite Galois extension $k_2$ of $k_1$
and every conjugacy class $C$ of $\Gal(k_2/k_1)$,
there is a prime $\mfq$ of $k_1$ which is unramified in $k_2$, for which
$\Frob_{\mfq}=C$,
for which $\N(\mfq)$ is a prime number, and which satisfies the bound
$\N(\mfq)\leq 2 d_{k_2}^{A_1}$.
Here $\Frob_{\mfq}$ is the (arithmetic) Frobenius conjugacy class
at $\mfq$ in $\Gal(k_2/k_1)$.

\end{thm}

\begin{prop}[{\cite[Proposition 4.2]{Da}}]
\label{boundofS}

Let $A_1$ be the constant in Theorem \ref{effectivechebotarev}.
Then we can take $\cSn$ so that every prime $\mfq\in\cSn$
satisfies $\N(\mfq)\leq 2d_k^{A_1 h_k}$.

\end{prop}

For a place $v$ of $k$ and an element $\alpha\in k$,
define $||\alpha||_v$ as follows.

\begin{itemize}
\item
If $v$ is finite, let $\mfq$ be the prime of $k$
corresponding to $v$,
%let
%$\val_{\mfq}:k^{\times}\twoheadrightarrow\Z$
%be the additive valuation associated to $\mfq$,
and let
$||\alpha||_v:=\N(\mfq)^{-\ord_{\mfq}(\alpha)}$
where $\ord_{\mfq}(\alpha)$
is the order of $\alpha$ at $\mfq$.
Here we put $||\alpha||_v:=0$ if $\alpha=0$.
\item
If $v$ is real, let $\tau:k\hookrightarrow \R$ be the embedding
corresponding to $v$,
and let $||\alpha||_v:=|\tau(\alpha)|$.
\item
If $v$ is complex, let $\tau:k\hookrightarrow \C$ be one of the embeddings
corresponding to $v$,
and let $||\alpha||_v:=|\tau(\alpha)|^2$.
\end{itemize}
For an element $\alpha\in k$, let $\rmH(\alpha)$
denote the absolute height of $\alpha$ defined by
$$\rmH(\alpha):=\left(\prod_v\max(1,||\alpha||_v)\right)^{\frac{1}{n_k}},$$
where $v$ runs through all places of $k$.
We know that there is a positive constant $\delta_k$, depending only on $k$,
such that for every non-zero element $\alpha\in k$ that is not a root of unity
we have $\log\rmH(\alpha)\geq\delta_k/n_k$
(\cf \cite[p.70]{BG}).
Fix such a constant $\delta_k$.
Define the constant
$C_1(k):=\frac{r_k^{1+r_k}\delta_k^{1-r_k}}{2}$.

\begin{prop}
\label{boundofgenerator}

Let $\mfq$ be a prime of $k$.
Then there is an element $\alpha'_{\mfq}\in\cO_k\setminus\{0\}$
which satisfies
$\mfq^{h_k}=\alpha'_{\mfq}\cO_k$
and
$\rmH(\alpha'_{\mfq})\leq|\Norm_{k/\Q}(\alpha'_{\mfq})|
^{\frac{1}{n_k}}\exp(C_1(k)R_k)$.

\end{prop}

\begin{proof}

%It follows directly from \cite[Lemme 3]{Da}.
%
Take an element $\gamma\in\cO_k\setminus\{0\}$
which satisfies
$\mfq^{h_k}=\gamma\cO_k$.
Then, by \cite[Lemme 3]{Da}, there is an element $u\in\cO_k^{\times}$
satisfying
$\rmH(u\gamma)\leq|\Norm_{k/\Q}(\gamma)|
^{\frac{1}{n_k}}\exp(C_1(k)R_k)$.
If we put $\alpha'_{\mfq}=u\gamma$, then
$\mfq^{h_k}=\alpha'_{\mfq}\cO_k$
and
$\rmH(\alpha'_{\mfq})\leq|\Norm_{k/\Q}(u^{-1}\alpha'_{\mfq})|
^{\frac{1}{n_k}}\exp(C_1(k)R_k)
=|\Norm_{k/\Q}(\alpha'_{\mfq})|^{\frac{1}{n_k}}\exp(C_1(k)R_k)$.
The last equality holds because 
$\Norm_{k/\Q}(u^{-1})\in\Z^{\times}=\{\pm 1\}$.

\end{proof}

Define the constant
$C_2(k):=\exp(24n_k C_1(k)R_k)$.

\begin{prop}
\label{boundofalphaepsilon}

Under the situation in Proposition \ref{boundofgenerator}, we have
$|(\alpha'_{\mfq})^{\varepsilon}|\leq\N(\mfq)^{24h_k}C_2(k)$
for any $\varepsilon\in\cE(k)$. %and any $\tau\in\Gal(k/\Q)$.

\end{prop}

\begin{proof}

Let
$\varepsilon=\sum_{\sigma\in\Gal(k/\Q)}a_{\sigma}\sigma$.
Then
$$|(\alpha'_{\mfq})^{\varepsilon}|
=\left|\prod_{\sigma\in\Gal(k/\Q)}(\alpha'_{\mfq})^{a_{\sigma}\sigma}\right|
\leq\left(\prod_{\sigma\in\Gal(k/\Q)}\max(1,|(\alpha'_{\mfq})^{\sigma}|)\right)^{24}
=\prod_{v|\infty}\max(1,||\alpha'_{\mfq}||_v)^{24}$$
$$=H(\alpha'_{\mfq})^{24 n_k}
\leq |\Norm_{k/\Q}(\alpha'_{\mfq})|^{24}\exp(24 n_k C_1(k) R_k)
=\N(\mfq)^{24 h_k}C_2(k).$$
Note that the third equality holds because
$\alpha'_{\mfq}\in\cO_k\setminus\{0\}$.

\end{proof}

For $a>0$, define the constant
$C(k,a):=(a^{24h_k}C_2(k)+a^{12h_k})^{2n_k}$.

\begin{prop}
\label{boundofnorm}

Under the situation in Proposition \ref{boundofgenerator}, we have
$|\Norm_{k(\beta_{\mfq})/\Q}((\alpha'_{\mfq})^{\varepsilon}-\beta_{\mfq}^{24 h_k})|
\leq C(k,\N(\mfq))$
for any $\varepsilon\in\cE(k)$ and any $\beta_{\mfq}\in\cFR(\N(\mfq))$.

\end{prop}

\begin{proof}

For any $\tau\in\Gal(k(\beta_{\mfq})/\Q)$,
we have
$|((\alpha'_{\mfq})^{\varepsilon}-\beta_{\mfq}^{24h_k})^{\tau}|
\leq |(\alpha'_{\mfq})^{\varepsilon\tau}|+|\beta_{\mfq}^{24 h_k\tau}|
\leq \N(\mfq)^{24 h_k}C_2(k)+\N(\mfq)^{12 h_k}$.
Then
$|\Norm_{k(\beta_{\mfq})/\Q}((\alpha'_{\mfq})^{\varepsilon}-\beta_{\mfq}^{24 h_k})|
=\prod_{\tau\in\Gal(k(\beta_{\mfq})/\Q)}
|((\alpha'_{\mfq})^{\varepsilon}-\beta_{\mfq}^{24h_k})^{\tau}|
\leq (\N(\mfq)^{24 h_k} C_2(k)+\N(\mfq)^{12 h_k})^{2 n_k}
=C(k,\N(\mfq))$.

\end{proof}

From now to the end of this section, take $\cSn$ as in Proposition \ref{boundofS},
and take $\alpha_{\mfq}$ to be $\alpha'_{\mfq}$ in Proposition \ref{boundofgenerator} for any $\mfq\in\cSn$.

\begin{prop}
\label{boundofM2k}

For any $m\in\cMn_2(k)$, we have $|m|\leq C(k,2 d_k^{A_1 h_k})$.

\end{prop}

\begin{proof}

We have
$m=\Norm_{k(\beta_{\mfq})/\Q}(\alpha_{\mfq}^{\varepsilon}-\beta_{\mfq}^{24 h_k})$
for some elements
$\mfq\in \cSn$, $\varepsilon\in\cE(k)$ and
$\beta_{\mfq}\in\cFR(\N(\mfq))$.
Then we obtain
$|m|\leq C(k,\N(\mfq))\leq C(k,2 d_k^{A_1 h_k})$
by Propositions \ref{boundofS} and \ref{boundofnorm}.

\end{proof}

\begin{thm}
\label{boundofN1k}

For any $p\in\cNn_1(k)$, we have $p\leq C(k,2 d_k^{A_1 h_k})$.

\end{thm}

\begin{proof}

Let $p\in\cNn_1(k)$.
If $p\in\cNn_0(k)$, then $p\leq C(k,2 d_k^{A_1 h_k})$.
If $p\in\cTn(k)$, then $p\leq \max(3,2 d_k^{A_1 h_k})$.
If $p\in\Ram(k)$, then $p\leq d_k$.
Therefore we conclude $p\leq C(k,2 d_k^{A_1 h_k})$.

\end{proof}

\def\bibname{References}

(Keisuke Arai)
Department of Mathematics, School of Engineering,
Tokyo Denki University,
5 Senju Asahi-cho, Adachi-ku, Tokyo 120-8551 Japan

\textit{E-mail address}: \texttt{araik@mail.dendai.ac.jp}

\end{document}